\documentclass[12pt,twoside]{amsart} 

\usepackage{amssymb,latexsym,bbm,graphicx,epsfig,epic,eepic,oldgerm,psfrag,comment}
\usepackage{amsmath,amsfonts,amscd,mathrsfs,amsthm}
\usepackage{mathtools}
\usepackage{mathtools}
\usepackage{forest}

\usepackage{tikz,wasysym}  \usetikzlibrary{trees}
\usepackage{marvosym}                  
\usepackage{hyperref} 

\usepackage{xypic} 
\input xy
\xyoption{all}

\theoremstyle{plain}
\newtheorem{theorem}{Theorem}[section]
\newtheorem{lemma}[theorem]{Lemma}                           
\newtheorem{proposition}[theorem]{Proposition}
\newtheorem{corollary}[theorem]{Corollary}
\newtheorem*{remark*}{Remark}
\newtheorem*{remarks*}{Remarks}
\newtheorem{remark}[theorem]{Remark}

\newtheorem{example}[theorem]{Example}

\newtheorem*{example*}{Example}
\newtheorem*{examples*}{Examples}

\newtheorem*{definition*}{Definition}

\newtheorem{question}[theorem]{Question}

\newtheorem*{claim*}{Claim}

\numberwithin{figure}{section}
\numberwithin{equation}{section}

\newcommand{\proofend}{\hspace*{\fill} $\square$\\}

\def\1{\:\!}
\def\2{\;\!}

\def\eps{\varepsilon}

\def\Diffc0{\operatorname{Diff^c_0}}

\def\Sympc0{\operatorname{Symp^c_0}}

\def\vol{\operatorname{vol}}

\def\CC{\mathbb{C}}

\def\NN{\mathbb{N}}
\def\OO{\mathbb{O}}

\def\RR{\mathbb{R}}

\def\ZZ{\mathbb{Z}}

\def\R{\operatorname{\mathbb{R}}}

\begin{document}

\title{\vspace*{0cm} On the large-scale geometry of domains in an exact symplectic 4-manifold}
\begin{abstract}

We show that the space of open subsets of any complete and exact symplectic $4$-manifold has infinite dimension with respect to the symplectic Banach-Mazur distance; the quasi-flats we construct take values in the set of dynamically convex domains.   In the special case of $\mathbb{R}^4$, we therefore obtain the following contrast: the space of convex domains is quasi-isometric to a plane, while the space of dynamically convex ones has infinite dimension.
In the special case of $T^* S^2$, a variant of our construction resolves a conjecture of Stojisavljevi\'{c} and Zhang, asserting that the space of star-shaped domains in $T^* S^2$ has infinite dimension.   Another corollary is that the space of contact forms giving the standard contact structure on $S^3$ has infinite dimension with respect to the contact Banach-Mazur distance.

%


\end{abstract}


\author{Dan Cristofaro-Gardiner and Richard Hind}

\date{\today}

\maketitle

\section{Introduction}

\subsection{Domains in $\mathbb{R}^4$}

Let $U$ and $V$ be open domains in $\mathbb{R}^4$, with its standard symplectic form.  We recall a definition introduced by Ostrover and Polterovich. The (coarse) {\em symplectic Banach-Mazur distance} is defined by 
\[ d_{SM}(U,V) = \inf \lbrace \ln(T) | U \to T \cdot V, V \to T \cdot U \rbrace.\]
Here, we mean the usual convention of scaling, where the capacity scales linearly with $T$; the arrow refers to a symplectic embedding.  

The symplectic Banach-Mazur distance is a kind of cousin of Hofer's metric.  The geometry of Hofer's metric has played a central role in modern symplectic topology, but many basic questions regarding the symplectic Banach-Mazur distance remain poorly understood.  Here, 
we will be interested in the large-scale geometry of this space.
More precisely, recall that a map $\psi: (X_1,d_1) \to (X_2,d_2)$ is called a {\em quasi-isometric embedding} if there are positive constants $A, B > 0$ such that
\[ 1/A d_1(x,y) - B \le d_2(x,y) \le A d_2(x,y) + B.\]
A {\em quasi-isometry} is a quasi-isometric embedding that is quasi-surjective, in the sense that every element in the target is a bounded distance from the image, and we call the {\em quasi-isometry} type of a space its equivalence class up to quasi-isometry.

We are interested here in the following question:

\begin{question}
\label{que:large}
How many linearly independent directions are there in the large scale geometry of the space of domains with respect to $d_{SM}$?
\end{question}

This is one of the first questions that one would like to know about the large-scale geometry.
However, it
has remained open for some time despite attracting interest from experts.
There is 
a variant of the symplectic Banach-Mazur metric, called the {\em fine symplectic Banach-Mazur distance}, where one demands in addition an unknottedness condition on the image of the embedding and for which the answer to Question~\ref{que:large} is known --- there are infinitely many such directions, in any dimension \cite{usher}.
However, the coarse case is a different matter that has been generally considered difficult and requiring of new ideas; see e.g. the discussion in  \cite[p. 11]{rosenzhang}, \cite[p. 235]{usher}. 

To make this precise, recall that a {\em rank $n$ quasi-flat} in a metric space $(X,d)$ is a quasi-isometric embedding $(\mathbb{R}^n,||_\cdot||_{\infty}) \to (X,d)$. The {\em quasi-flat rank} of $(X,d)$ is the supremum, over $n$, such that a rank $n$ quasi-flat in $(X,d)$ exists.   
Our main result answers Question~\ref{que:large}:

\begin{theorem}
\label{thm:main}
The space of star-shaped domains in $\mathbb{R}^4$, with respect to the pseudo-metric $d_{SM}$, has infinite quasi-flat rank.
\end{theorem}


\begin{remark}\label{fine} 
\normalfont
Our embeddings of $\RR^N$ into the space of star-shaped domains are in fact also quasi-flats for the fine symplectic Banach-Mazur metric $d_f$, that is, the two metrics are comparable on the image. Indeed, by definition we have $d_{SM} \le d_f$. But in the proof our upper bounds on $d_{SM}$ come from inclusions -- which of course are unknotted -- so also give upper bounds for $d_f$.
\end{remark}

In symplectic geometry, many results are known for convex domains, but the role of convexity, which is not preserved by symplectomorphisms, is still not well-understood.   Our theorem highlights the extent to which convexity is special from the large-scale geometric point of view:

\begin{example}
\normalfont
If one restricts to the space of {\em convex} domains, 
then the space is quasi-isometric to a plane, via the John's ellipsoid construction; the same is true in higher dimensions due to a construction of Guth \cite{guth}.  Thus, our theorem gives a sharp contrast between the convex and star-shaped case, with the former appearing very sparse in the latter.  In fact, we will see in our construction that our quasi-flats can be assumed to have smooth, dynamically convex, boundary, thus the space of dynamically convex domains also has infinite dimension.
\end{example}

\begin{remark}
\label{rem:newproof}
\normalfont
By the discussion in the above example, our theorem therefore recovers an important theorem of Chaidez-Edtmair \cite{ched}, showing that dynamically convex domains need not be convex; dynamical convexity is a symplectic condition implied by convexity, and we refer to \cite{ched} for the definitions and for further context.   In fact, the full force of Theorem~\ref{thm:main} is not needed for this, and we give a shorter proof for the interested reader in Theorem~\ref{thm:ched} below.
Another proof of the Chaidez-Edtmair theorem, via considerations of the metric $d_{SM}$, appears in \cite{dgvz}.  
\end{remark}

\begin{remark}
\label{rem:notsame}
\normalfont

One can attempt to obtain a more refined understanding of the large-scale geometry.   To give one example, in the definition of the coarse Banach-Mazur distance we can replace symplectic embeddings simply by inclusions, obtaining another pseudo-metric $d_I$ which bounds $d_{SM}$ from above. 
Optimistically, one might hope that the identity map gives a quasi-isometry between the space of domains equipped with one pseudo-metric or the other.
It can be checked that 
this is true for the subset of convex toric domains, and the proof of Theorem 1.2 implies that this is also true
on the image of any one of our quasi-isometric embeddings from $\R^n$ into the space of concave toric domains.  However, we show in Proposition~\ref{prop:notsame} that this is not true in general, even if one restricts to concave toric domains.

\end{remark}

\subsection{Exact symplectic manifolds}

Theorem~\ref{thm:main} admits a generalization to an arbitrary exact symplectic $4$-manifold, which we now explain.
 
Let $(M, \lambda)$ be a manifold with a $1$ form such that $\omega = d \lambda$ is symplectic, and such that the flow of the (Liouville) vector field $V_{\lambda}$ exists for all positive time; we call this a {\em complete exact symplectic $4$-manifold}. Here $V_{\lambda}$ is defined by $V_{\lambda} \lrcorner \omega = d \lambda$. We write $\{ \phi^t \}_{t \ge 0}$, for the corresponding $1$ parameter family of embeddings, so ${\phi^t}^* \omega = e^t \omega$. Given an open subset $U \subset M$ and $T \ge 1$ we define $T \cdot U := \phi^{ln T}(U)$.  As an example, we can set $(M, \lambda)$ to be $\mathbb{R}^4$ with the form $\lambda = \frac12 \sum_{i=1}^2 (x_i dy_i - y_i dx_i)$. Then $V = \sum_{i=1}^2 (x_i \partial{x_i} + y_i \partial{y_i})$ is the radial vector field and $T \cdot U$ corresponds to scaling as before.

 
Then the Banach-Mazur distance is defined on open subsets of such $M$ by
\begin{equation}\label{BMexact} 
d^{(M, \lambda)}_{SM}(U,V) = \inf \lbrace \ln(T) | U \to T \cdot V, V \to T \cdot U \rbrace
\end{equation}
where now the arrows denote the existence of a Hamiltonian diffeomorphism of $M$.  (Note that we are allowing for the possibility that the distance between $U$ and $V$ could be infinite.) 
In the case of the standard example of $\mathbb{R}^4$ this coincides with the definition of the coarse symplectic Banach-Mazur distance used earlier in the paper.  


\begin{theorem} \label{exactflats}
Let $(M, \lambda)$ determine a (complete) exact symplectic manifold as above. Then the open subsets of $M$, with respect to the pseudo-metric $d^{(M, \lambda)}_{SM}$, has infinite quasi-flat rank.
\end{theorem}

\subsection{Stojisavljevi\'{c} and Zhang's conjecture}

Although Theorem~\ref{exactflats} applies to any complete exact symplectic manifold, it is interesting to impose further conditions on the open subsets under consideration.  A notable example of this is the following  conjecture of  Stojisavljevi\'{c} and Zhang; we can resolve it with our methods.  

To elaborate, call an open domain of an exact symplectic manifold {\em star-shaped} if its boundary is transverse to  
$V_{\lambda}$.
In the case $T^* S^2$, the star-shaped sets are the same as 
subsets 
which are starshaped about the origin in each fiber. Given a subset $U \subset T^* S^2$ and $T>0$ it is natural to define $T \cdot U = TU$ where multiplication is defined fiberwise. We note that this fiberwise multiplication satisfies $T^* \lambda = T\lambda$ where $\lambda = pdq$ is the canonical Liouville form. Then we define a Banach-Mazur distance on starshaped subsets by
 \[ d^{T^* S^2}_{SM}(U,V) = \inf \lbrace \ln(T) | U \to T \cdot V, V \to T \cdot U \rbrace. \]
In other words, this is just the restriction of \eqref{BMexact} to starshaped domains, in the case when $(M, \lambda)$ is $T^* S^2$ with its Liouville form.

Here is our result:

\begin{theorem}\cite[Conj. 1.13]{sz} \label{flatsS2}
The starshaped open subsets of $T^* S^2$, with respect to the pseudo-metric $d^{T^* S^2}_{SM}$, has infinite quasi-flat rank.
\end{theorem}

\begin{remark}
\normalfont
The above theorem settles Stojisavljevi\'{c} and Zhang's Conjecture 1.13 from \cite{sz}.   Stojisavljevi\'{c} and Zhang prove the analogous result for $T^*S$ when $S$ is a closed orientable surface with positive genus in \cite{sz}, and they show that in the genus zero case, there is a quasi-isometric embedding of a ray.   This is perhaps reminiscent of the situation with respect to Hofer's metric on surfaces, where new techniques
were used in the genus zero case compared to the case of higher genus; see \cite{jems, ps}.  The main concern in Stojisavljevi\'{c} and Zhang's paper was the fine Banach-Mazur distance, but our results on the coarse distance also imply quasi-flats for the fine, see Remark \ref{fine}.
\end{remark}

\subsection{The contact Banach-Mazur distance}

There is an analogue of the symplectic Banach-Mazur distance in contact geometry, more precisely for the space of contact forms on a fixed contact structure, called the {\em contact Banach-Mazur distance}, denoted $d_{CBM}$; we will not state the definition here for brevity, referring the reader instead to \cite[Defn. 4]{rosenzhang}.  It would be interesting to prove an analogue of Theorem~\ref{exactflats} in this setting, i.e. for contact structures on a closed three-manifold.  Since the quasi-flats we construct in Theorem \ref{thm:main} are described by starshaped domains whose boundary contact forms give the standard contact structure on $S^3$, we have the following corollary.

\begin{corollary}
The space of contact forms giving the standard contact structure on $S^3$ has infinite quasi-flat rank with respect to $d_{CBM}$.
\end{corollary}

Indeed, by \cite{rosenzhang}, Lemma 8, the metric $d_{CBM}$ on the boundary forms is bounded below by the metric $d_{SM}$ applied to our starshaped domains. However our upper bounds on $d_{SM}$ are derived from explicit symplectic embeddings, namely inclusions, which imply the same upper bounds for $d_{CBM}$. In other words our quasi-flats mapping $\RR^N$ into the space of starshaped domains can be composed with the map taking starshaped domains to contact forms on $S^3$, then the resulting composition is a quasi-flat in the space of contact forms.  For previous work on the contact Banach-Mazur distance, we refer the reader to \cite{rosenzhang, melistas}.

\subsection*{Acknowledgements}

Key discussions about this project occurred while the authors were in residence in Cortona, Italy to give the 2023 mini-course ``Symplectic geometry, capacities and embeddings."   We thank the Scuola Normale Superiore di Pisa, the Consorzio Interuniversitario per l’Alta Formazione in Matematica (CIAFM) and the Italian Ministry of Education University and Research
(MIUR) for their support.  We also thank Oliver Edtmair, Leonid Polterovich, Mike Usher, Luya Wang and Jun Zhang for very helpful comments. Finally DCG thanks the National Science Foundation for their support under agreement DMS-2227372, and RH thanks the Simons Foundation for their support under grant no. 633715.

\section{Proofs}

\subsection{ECH capacities}
\label{sec:ECH}

We begin by reviewing the main tool which will power our arguments.

Our obstructions come from a family of numerical invariants
\[ 0 := c_0(X,\omega) \le c_1(X_\omega) \le \ldots \le \infty,\]
defined for four-dimensional symplectic domains, called {\em ECH capacities}.  The ECH capacities were defined in \cite{h1}.  They are measurements of symplectic size: the key facts that we need to know about them here are the {\em Monotonicity Axiom}, which states that
\[ c_k(X_1,\omega_1) \le c_k(X_2,\omega_2),\]
for each $k$ whenever there is a symplectic embedding of $(X_1,\omega_1)$ into $(X_2,\omega_2)$, and the Scaling Axiom, which states that
\[ c_k(X, c \cdot \omega) = c \cdot (X,\omega),\]
for positive real numbers $c$.  Another property of ECH capacities which is at the heart of our construction is the Weyl Law, which states that
\[  \lim_{k \to \infty} \frac {c_k(X,\omega)^2}{k} = 4 \vol(X,\omega);\]
this was proved for domains in $\mathbb{R}^4$, which is the case relevant here, in \cite{h1}, and for more general domains in \cite{cghr}.

We will also need to know how to calculate ECH capacities for the following family of domains.  Recall that a {\em toric domain} is a region
\[ X_{\Omega} = \lbrace (z_1,z_2) | \pi (|z_1|^2, |z_2|^2) \in \Omega \rbrace \subset \CC^2\]
associated to a subset $\Omega \subset \mathbb{R}^2$ in the first quadrant.  A toric domain is called {\em concave} if $\Omega$ is bounded by the axes and a convex function $f: [0,A]\to [0,B]$ with $f(0) = B$ and $f(A) = 0$.   Any concave toric domain has a canonical ball-packing called its {\em weight expansion}: this is the ``greedy" ball-packing via inclusion in the moment image, and we refer the reader to \cite{ccghr} for its definition.  The size of these balls are called the {\em weights} of the concave toric domain.  As proved in \cite{ccghr}, the weights $(a_1, \ldots,)$ of a concave toric domain $X$ determine its ECH capacities via the rule
\[ c_k(X) = max_{k_1 + \ldots + k_m = k} c_{k_i}(B(a_i)).\]
We also recall here that the sequence of ECH capacities $c_{ECH} = (c_1,c_2,\ldots)$ of the ball are given by the explicit formula:
\[ c_{ECH}(B(a)) = (0,a,a,2a,2a,2a,\ldots).\] 
For more detail about the above, we refer the reader to \cite{ccghr}.

ECH capacities are defined using a Floer theory for contact three-manifolds called embedded contact homology.  For our purposes here, one could alternatively use the ``elementary ECH capacities" defined in \cite{h2}; these can be defined without requiring Floer theory.  Since our arguments use only some formal properties shared by both invariants, either will suffice.

\subsection{A dynamically convex domain that is not convex}

Before proceeding with our argument, we now give the promised new proof of Chaidez-Edtmair's theorem; see the discussion in Remark~\ref{rem:newproof}.  
Our proof is also a ``warm-up" for the general case, in that it introduces the key idea of ``moving volume" to the latter end of the ECH capacities; for more about this, see Remark~\ref{rem:idea}.



\begin{theorem}\cite{ched} There exists a smooth and dynamically convex domain that is not convex
\label{thm:ched}
\end{theorem}

\begin{proof}Consider a smooth concave toric domain $X$, whose moment polytope has the following properties: it starts at $(0,1)$ with a line of slope $-1$, connecting to $(1-\eps, \eps)$, and then has volume $V$.  We claim that for appropriate choices of $V, \eps > 0$, the domain $X$ can not be convex; since $X$ is a four-dimensional monotone toric domain, it is dynamically convex, by \cite[Prop. 1.8]{ghr}.

Assume that $X$ is convex.  Then, by John's theorem, there is some universal constant $C$ (i.e. independent of everything) and some ellipsoid $\lambda E(1,a)$ with $a \ge 1$ such that 
\begin{equation}
\label{eqn:inclusions}
(1/C) \lambda E(1,a) \subset X \subset \lambda E(1,a). 
\end{equation}  
From above, we know that
\begin{equation}
\label{eqn:inequalities}
\lambda \ge 1, \quad \quad  (1/C^2) \lambda^2 a \le V, \quad \quad V \le \lambda^2 a,
\end{equation}
where in the first inequality we have used the fact that the Gromov width of $X$ is $1$.

It follows from \eqref{eqn:inequalities} that $a \le C^2 V / \lambda^2 \le C^2 V$.  
Thus, $a$ has an $\eps$-independent upper bound, depending only on $V$.   It follows that, there exists $k_0$, depending only on $V$ and not on $\eps$, such that  
\[ c_{k_0}((1/C) \cdot \lambda E(1,a)) \ge  (\lambda/C) \sqrt{a} \sqrt{k_0} \ge  (1/C) \sqrt{V} \sqrt{k_0},\]
where in the first inequality we have used the Weyl Law for ellipsoids, and in the last inequality we have used \eqref{eqn:inequalities}.  
On the other hand, if $\eps$ is sufficiently small with respect to $k_0$, $c_{k_0}(X) \le c_{k_0}(B(1)) + 1$, since all the weights in $X$ that are not $1$ are bounded by $\eps$.  This is a contradiction for suitable choice of $V$ in view of the first embedding in \eqref{eqn:inclusions} by the Weyl Law\footnote{We note that the Weyl Law in the cases relevant here, i.e. the case of an ellipsoid, follows from direct computation once one knows the ECH capacities of an ellipsoid, as in \cite[Exer. 1.4]{hutchings}.}, applied to $B(1)$.
\proofend

\end{proof}

\subsection{Change of variables}

We now proceed with our arguments.  Let $Q_n: = \mathbb{R}^n_{> 0}$ denote the space of vectors with positive entries.  This has a metric 
\[ || x , y || = max(| ln (x_i/y_i) |). \]
The purpose of this section is to prove the following lemma, which describes some very useful quasi-flats in $Q_{2n}$. 

\begin{lemma}
There is a quasi-isometric embedding 
\[ (\mathbb{R}^{n},||\cdot||_{\infty}) \to \lbrace (B_1, \ldots, B_{2n}) | \hspace{1 mm} 8^2 \le B_k^2 \le B_{k+1} \hspace{1 mm} \forall k \rbrace \subset (Q_{2n},||\cdot|| ). \]
\end{lemma}

\begin{proof}

The proof comes from composing three quasi-isometric embeddings.

To elaborate, there is a quasi-isometric embedding
\begin{equation}
\label{eqn:known}
(\mathbb{R}^n,||\cdot||_{\infty}) \to (Q_{2n},||\cdot||_{\infty}),
\end{equation}
see \cite{sz}.  By composing with a translation of $Q_{2n}$, we can assume that any $(x_1,\ldots,x_{2n})$ in the image of \eqref{eqn:known} satisfies $x_i \ge 3$; this will be convenient in what follows. 

Next, there is also a quasi-isometric embedding
\begin{equation}
\label{eqn:linear} 
(Q_{2n},||\cdot||_{\infty}) \to  (Q_{2n},||\cdot||_{\infty}),
\end{equation}
\[ (x_1,\ldots,x_{2n}) \to (y_1, \ldots, y_{2n} )\]
where 
\[ y_1 = x_{2n} + \ldots + x_{1},\]
\[ y_2 = 3 x_{2n} + 2 x_{2n-1}\ldots + 2 x_{1},\]
\[  \ldots\] 
\[ y_{i+1} = (2^{i} + 2^{i-1}) x_{2n} + (2^i + 2^{i-2}) x_{2n-1} + \ldots + 2^i x_{2n-i} + \ldots + 2^i x_1.\]

This is a quasi-isometry, in fact, bi-Lipschitz, since it is an invertible linear map; it is invertible because
we can recover each $x_i$ with $i \ge 2$ via the rules $x_{2n} = y_2 - 2 y_1, x_{2n-1} = y_3 - 2 y_2, \ldots,$ and then $x_1 = y_1 - x_2 - \ldots - x_{2n}.$  The key property that it satisfies for our purposes is that each 
\begin{equation}
\label{eqn:key}
y_{i+1} \ge 2 y_i \ge 12.
\end{equation}

Finally,
note the map 
\begin{equation}
\label{eqn:exponent}
(Q_{2n}, || \cdot ||_{\infty} ) \to (Q_{2n}, || \cdot ||), \quad (y_1, \ldots, y_{2n}) \to (e^{y_1}, \ldots, e^{y_{2n}}),
\end{equation}
is, by the definition of the metric $|| \cdot ||$, an isometry.

Let us now put all of this together.  Apply first \eqref{eqn:known}, then \eqref{eqn:linear}, and finally \eqref{eqn:exponent}.  This is a quasi-isometric embedding, and it satisfies $8^2 \le B_k^2 \le B_{k+1}$ by \eqref{eqn:key}.

 \proofend

\end{proof}

\subsection{Proof of Theorem~\ref{thm:main}}
\label{sec:build}

We now prove Theorem~\ref{thm:main}; the proof is the content of this section.

In view of the above quasi-flats in $Q_{2n}$, it suffices to consider parameters $(B_1, \cdots , B_N)$ satisfying $8 \le B_k^2 \le B_{k+1}$ for all $k$; these inequalities will be assumed in what follows.  Also, the quasi-isometry type is unchanged if we further restrict to the case when all $B_k^{k+1} \in \NN$; 
this restriction will also be our standing assumption in what follows.

Throughout we let $C$ denote a constant independent of the $B_k$ but which may depend on $N$. Also write $a \sim b$ if there exists such a $C$ with $a/C \le b \le aC$; write $a \sim \hspace{1 mm} \ge b$ and $a \sim \hspace{1 mm} \le b$ if the corresponding inequalities hold.

Let us now explain the construction of our quasi-flats.  

\begin{remark} 
\label{rem:idea} 
\normalfont
({\em Idea of the proof.})  Here we explain the basic idea of the construction.  Recall that the Weyl Law
reviewed above
asserts that the ECH capacities $c_k$ of a domain grow like $\sqrt{k}$, with coefficient detecting the volume.    The idea is then to construct a domain that realizes the different parameters $B_i$ as volume coefficients on different intervals $k$: the rough idea is to start with a domain of volume $B_1$, solve for the value of $k$ for which the asymptotics of $B_1$ start to dominate, say at $k \ge k_0$, then add to it a piece that increases the volume without affecting any $c_k$ for $k \le k_0$, and to continue in this manner. 
\end{remark}

In \cite[Ex. 4.8]{CG}, a procedure is described to assign to any finite set of positive numbers a connected concave toric domain.   We apply this procedure to associate to the parameters $B_i$ from above
the 
concave toric domain derived from the disjoint union
of ellipsoids 
$$\hat{X} = \cup_{i=1}^N E(B_i^{-i/2}, B_i^{1+i/2}),$$ 
in other words, we decompose each of the ellipsoids into their weight sequence, and then concatenate these weights into a concave toric domain following \cite[Ex. 4.8]{CG}.  
Let us record in particular that it then follows that the weights of $\hat{X}$ are of the form 
\[ a_i = B_i^{-i/2}\]
with $a_i$ appearing  $B_i^{i+1}$ times.

We sometimes write the domain $\hat{X}_v$ when we want to be explicit about the dependence on the parameters $v = (B_1,\ldots, B_N)$.

Our first order of business is to smooth the domains $\hat{X}$, and obtain upper bounds on the Banach-Mazur distance.  In fact, our upper bounds will come from scaling, so all of this is elementary and collected in the following lemma.

\begin{lemma}
\label{lem:upperbound}
Let $v = (B_1,\ldots, B_N)$ and $w = (B'_1,\ldots, B'_N)$.

There is a smoothing of the domain $\hat{X}_v$ to a concave toric domain $X_v$ 
containing $\hat{X}_v$
with smooth boundary and the following properties:
\begin{itemize}
\item The weights of $X_v$ contain the weights of $\hat{X}_v,$ all other weights have size $ < B^{-N}_N,$
 and $vol(X_v) \le vol(\hat{X}_v) + 1$.  
\item $d_{SM}(X_v,X_w) \le (N/2 + 1) ||v,w|| + 1$ 
\end{itemize}

\end{lemma} 

\begin{proof}
Let $m := e^{||v,w||}$.  We first note that 
\begin{equation}
\label{eqn:inclusion}
\hat{X}_v \subset m^{N/2 + 1} \cdot \hat{X}_w, \quad \hat{X}_w \subset m^{N/2 + 1} \cdot \hat{X}_v.
\end{equation}
Indeed, by definition, each $E(B_i^{-i/2}, B_i^{1+i/2}) \subset m^{N/2 + 1} E(B_i^{'-i/2}, B_i^{'1+i/2})$, with the corresponding containment also holding with $B'_i$ and $B_i$ reversed, and then this follows from the construction procedure from \cite[Ex. 4.8]{CG}.


Now it follows from the definition of the weight expansion that we can approximate $\hat{X}_v$ and $\hat{X}_w$ by the smooth domains $X_v$ and $X_w$, such that the first item of the lemma holds: to do this, we just add to the upper boundary of $\hat{X}_v$ (resp. $\hat{X}_w$) sufficiently small smooth curves connecting any two consecutive line segments.  Moreover, these approximations can be assumed Hasudorff close to $\hat{X}_v$ (resp. $\hat{X}_w$), and so the second item follows.
\proofend


\end{proof}


It remains to show the lower bound, which will take up the remainder of this section. More precisely, we aim to prove the following:

 \begin{proposition}\label{bm} $$d_{SM}(X_v, X_w) \sim \hspace{1 mm} \ge \max_i | \log (\frac{B_i}{B'_i}) |.$$
\end{proposition}

Note that Theorem~\ref{thm:main} is an immediate consequence of Proposition~\ref{bm} and Lemma~\ref{lem:upperbound}.

As preparation for the proof, we suppose $\max_i | \log (B_i / B'_i) | = \log (B'_M / B_M)$.  If $M = N$, then volume considerations suffice for Proposition~\ref{bm}, so we assume $M < N$; this will be our standing assumption for the remainder of the note.  Then it suffices to find a $k$ with $\log (c_k(X') / c_k(X)) \sim \log( B'_M / B_M)$, where the $c_k$ are the ECH capacities from \ref{sec:ECH}.

Our first key observation is that
our assumption on the maximum implies 
\begin{equation}
\label{eqn:fromassumption}
B'_M \le B_{M+1},
\end{equation} 
since otherwise we have $$\frac{B'_{M+1}}{B_{M+1}} > \frac{B'_{M+1}}{B'_{M}} \ge B'_{M} \ge \frac{B'_{M}}{B_{M}}$$ where the second and third inequalities apply our assumptions that $B'_{M+1} \ge {B'_M}^2$ and $B'_M \ge 1$ respectively. This gives a contradiction.

We now seek to establish the following Lemma:

\begin{lemma}\label{cap} Assume that $M < N$ and suppose $B_M^{M+1} \le k \le B_{M+1}^{M+1}$. Then $c_k(X) \sim k^{1/2} B_M^{1/2}$.
\end{lemma}

We can think of the lemma as asserting that the $c_k$ in the above range see only the volume of the first $M$ pieces of $X$; see Remark~\ref{rem:idea}.

Before giving the proof, 
let us explain why the lemma proves the proposition:

\vspace{1 mm}

{\em Proof of Proposition~\ref{bm}, assuming Lemma \ref{cap}.}  Recall our standing assumption that $M < N.$  We choose $k = {B'_M}^{M+1}$. Then by Lemma \ref{cap} we have $c_k(X') \sim k^{1/2} {B'_M}^{1/2}$.
On the other hand, by \eqref{eqn:fromassumption} we have $B_M \le B'_M \le B_{M+1}$. Hence by Lemma  \ref{cap} again we have $c_k(X) \sim k^{1/2} {B_M}^{1/2}$ and then Proposition~\ref{bm} follows.  
\proofend


Thus, it remains to prove Lemma \ref{cap}, and this will take up the remainder of the file.

Let  $n_i = B_i^{1+i} \in \NN$ for all $i$. Then setting $a_i = B_i^{-i/2}$ we have the weight decomposition for $\hat{X}$: $$\sqcup_i B(a_i)^{ \times n_i}.$$ 
There is a similar weight decomposition for $X$, where the additional weights are all smaller than $a_N$, by Lemma~\ref{lem:upperbound}.
We recall the formula for the ECH capacities for a concave toric domain in terms of its weight sequence; see \ref{sec:ECH}. In the case of $X$ this gives 
\begin{equation}
\label{eqn:ck}
c_k(X) = \max\{ \sum_{i,j} d_{i,j} a_i \, | \, 1 \le j \le n_i, \sum_{i,j} (d_{i,j}^2 + d_{i,j}) \le 2k \},
\end{equation}
where the $d_{ij}$ are nonnegative integers; it might be helpful to note that the  $d_{ij}$ for $1 \le i \le N$ correspond to the weights of $\hat{X}$, while the higher index $d_{ij}$ come from the further smaller weights introduced in the smoothing.  We call a sequence of $d_{i,j}$ satisfying the necessary bound a {\em candidate sequence}, and a candidate sequence actually realizing the max is called an {\em optimizing sequence}.

To prove Lemma~\ref{cap}, we will first prove the following lemma, 
which essentially asserts that for $k$ in a certain range, only the first $P$ pieces of $X$ contribute to $c_k$.

\begin{lemma}\label{extra} Suppose $P < N,$ 
$B_P \ge 8$, 
and $k \le  B_{P+1}^{P+1}$. Then the optimizing sequence for $c_k(X)$ has 
\[
\sum_{i>P, j} d_{i,j} \le \sqrt{ \frac{8k}{B_P^{P+1}} } . 
\]
In particular the contribution of these terms to $c_k(X)$ is at most $\sqrt{8k} B_P^{-P/2 -1/2} B_{P+1}^{-P/2 - 1/2}$.
\end{lemma}

\begin{proof} Setting $d = \sqrt{ \frac{2k}{n_P} } $, note that $n_P(d^2 + d) > 2k$ and hence any candidate sequence has $d_{P, j} < d$ for some $j$, say $j = j_0$.

Now, suppose that $\sum_{i>P, j} d_{i,j} > 2d$ 
in an optimizing sequence.   Then we can form a new sequence as follows: we reduce a particular positive $d_{ij}$ with $i > P$ by $1$, iterate this procedure $2d_{P, j_0}$ times, 
and then finally replace 
$d_{P, j_0}$ by $d_{P, j_0} +1$. This does not increase the sum $\sum_{i,j} (d_{i,j}^2 + d_{i,j})$ and so we still get a candidate sequence. However the sum $\sum_{i,j} d_{i,j} a_i$ changes by
at least
$$a_P - 2d_{P, j_0} a_{P+1} > B_P^{-P/2} - 2 \sqrt{ \frac{2k}{B_P^{P+1}} } B_{P+1}^{-P/2 - 1/2},$$ 
$$\ge B_P^{-P/2} - 2 \sqrt{ \frac{2}{B_P^{P+1}} }.$$ 
where in the last inequality we have used the fact that we are assuming $k \le B^{P+1}_{P+1}.$  Since  $B_P^{-P/2} - 2 \sqrt{ \frac{2}{B_P^{P+1}}} \ge 0$ by our assumption on $B_P$, this is a contradiction to the assumption we started with an optimizing sequence. 

\proofend
\end{proof}

We can now give the promised proof of Lemma~\ref{cap}:

\begin{proof} 
{\em Proof of Lemma~\ref{cap}.}  

We use the formua \eqref{eqn:ck}.  By Lemma \ref{extra}, the contribution of the terms $d_{M+1, j}$ is at most $\sqrt{8} B_M^{-M/2 -1/2}$. 
Therefore, to get an upper bound of the correct order, we can optimize only over sequences with all $d_{M+1, j}=0$. For the upper bound we remove the requirement $d_{i,j} \in \ZZ$, so that we are optimizing over an ellipsoid, and then 
the maximum occurs when $d_{i,j} = d_i$ where $d_i = \lambda a_i - 1/2$ and $\sum_i n_i(d_i^2 + d_i) = 2k$.

Substituting for $d_i$ in the sum we get
$$\sum_{i=1}^M n_i(\lambda^2 a_i^2 - 1/4) = 2k$$ and so
$$\lambda^2 \sum^M_{i=1} B_i = 2k + \frac{1}{4} \sum^M_{i=1} B_i^{1+i}.$$

Now $\sum B_i \sim B_M$ and  $\sum B_i^{1+i} \sim B_M^{1+M}$. Given our bounds on $k$ we can therefore first absorb the   $\sum B_i^{1+i}$ into the $k$ term on the right hand side of the above equation, and then conclude that $\lambda^2 \sim k B_M^{-1}$. Hence $\lambda a_i \sim \sqrt{k} B_i^{-i/2}B_M^{-1/2}$ which
implies $d_i  \sim \hspace{1 mm} \le \sqrt{k} B_i^{-i/2}B_M^{-1/2}$.
Plugging this into the capacity estimate we therefore get
$$\sum d_{i,j} a_i = \sum n_i d_i a_i  \sim \hspace{1 mm} \le \sqrt{k} \sum^M_{i=1} B_i B_M^{-1/2} \sim \sqrt{k B_M},$$
which gives the upper bound.

For a lower bound we simply set $$d_{M, j} = d_M = \lfloor \sqrt{ \frac{2k}{n_M} } \rfloor \sim \sqrt{k} B_M^{-M/2 - 1/2}$$ and all other $d_{i,j} =0$.  This is a candidate sequence for $c_k(X)$ and gives the estimate $c_k(X) \sim \ge  \sqrt{k B_M}$ as required.
\proofend
\end{proof}

\subsection{Proof of Theorem~\ref{exactflats}}

Next we prove our generalization to exact symplectic manifolds, Theorem~\ref{exactflats}.

\begin{proof} Fixing a natural number $n$, the proof of Theorem  \ref{thm:main} defined a family of open subsets $X^v$ of $\mathbb{R}^4$, parameterized by
\[ v \in \tilde{Q} := \lbrace (B_1, \ldots, B_{2n}) | \hspace{1 mm} 8^2 \le B_k^2 \le B_{k+1} \hspace{1 mm} \forall k \rbrace \subset Q_{2n}. \]
This family of subsets was shown to admit quasi isometric embeddings of $\mathbb{R}^n$ with respect to the course symplectic Banach-Mazur distance $d_{SM}$.

Let $d = d^{(M, \lambda)}_{SM}$. Then it suffices to produce a family of open subsets $Y^v \subset M$, again for $ v \in \tilde{Q}$, with the property that 
$$d(Y^v, Y^w) \sim d_{SM}(X^v, X^w).$$

Fix a symplectic embedding $\psi: B( \epsilon) \to M$, where $B(\epsilon) \subset \R^4$ is the round ball of capacity $\epsilon$. 
Given $ v \in \tilde{Q}$ we define $T(v) = \frac{1}{\epsilon} \sum B_k^{k+1}$. We check that $\frac{1}{T(v)} \cdot X^v \subset B(\epsilon)$. Thus we can define
$$Y^v = T(v) \cdot \psi( \frac{1}{T(v)} \cdot X^v ).$$ So $Y^v$ is an open subset of $M$ symplectomorphic to $X^v$, and for any $T>0$ we have $T \cdot Y^v$ symplectomorphic to $T \cdot X^v$.

As Hamiltonian diffeomorphisms generate a subset of general symplectic embeddings, by definition we see
$$d(Y^v, Y^w) \ge d_{SM}(X^v, X^w).$$

For the opposite inequality, we recall that the proof of Theorem \ref{thm:main} showed that in fact
$$d_{SM}(X^v, X^w) \sim D(X^v, X^w)$$
where for subsets of $U, V \subset \R^4$ we set
\[ D(U,V) = \inf \lbrace \ln(T) | U \subset T \cdot V, V \subset T \cdot U \rbrace. \]

Hence it suffices to show $d(Y^v, Y^w) \le D(X^v, X^w).$ For this, let $\ln(M) > D(X^v, X^w)$, that is, assume $X^v \subset M \cdot X^w$ and $X^w \subset M \cdot X^v$, and we aim to find a symplectic isotopies mapping $Y^v$ into $M \cdot Y^w$ and $Y^w$ into $M \cdot Y^v$. The situation is symmetric in $v$ and $w$ so we focus on the first case.

Define a function $f(t)$ for $0 \le t \le 2$ by $f(t) = T(v)( (1-t) + tM)$ when $0 \le t \le 1$ and by $f(t) = M( (2-t)T(v) + (t-1)T(w))$ when $1 \le t \le 2$.

Then for $0 \le t \le 2$ we define $$Y^t = f(t) \cdot \psi( \frac{1}{f(t)} \cdot X^v)$$ and claim this is a symplectic isotopy as required.

First we check that $Y^t$ is well defined, that is, $\frac{1}{f(t)} \cdot X^v \subset B(\epsilon)$. For this, when $t \le 1$ we have $f(t) \ge T(v)$ and so the statement follows from the definition of $T(v)$. When $1 \le t \le 2$ we have $$\frac{1}{f(t)} \cdot X^v = \frac{1}{(2-t)T(v) + (t-1)T(w)} \cdot \frac{1}{M} \cdot X^v.$$ As $\frac{1}{M} \cdot X^v \subset X^v \cap X^w$ we have that both $\frac{1}{T(v)} \cdot \frac{1}{M} \cdot X^v$ and $\frac{1}{T(w)} \cdot \frac{1}{M} \cdot X^v$ lie in $B(\epsilon)$ and the conclusion follows.

We have $Y^0 = Y^v$ by definition, and can write $$Y^2 = M \cdot (T(w) \cdot \psi (\frac{1}{T(w)} \cdot \frac{1}{M} \cdot X^v).$$ But as $\frac{1}{M} \cdot X^v \subset  X^w$ the set $Y^2$ lies in $M \cdot Y^w$ as required.
\proofend
\end{proof}

\subsection{Proof of Theorem~\ref{flatsS2}}

We now prove the conjecture of Stojisavljevi\'{c} and Zhang; that is, we give the proof of Theorem~\ref{flatsS2}.

\begin{proof} Fix the function $R := \sum \pi |z_i|^2$ on $\CC^2$  and identify $S^3$ with $\{R=1\}$. Then $S^3$ is a contact manifold with contact form $\alpha := \frac 12 \sum (x_i dy_i - y_i dx_i)|_{S^3}$.

Starshaped domains in $\CC^2$ can be defined by functions $f : S^3 \to (0, \infty)$ by setting $$U_f := \{ Rz \, | \, z \in S^3, \, R< f(z) \}.$$

Given a round metric on $S^2$ we can define a contact type hypersurface  $\Sigma \subset T^* S^2$ by $\Sigma = \{ |p|=1 \}$ where $|p|$ is defined using the metric. This has contact form $\beta:= \lambda|_{\Sigma}$ where $\lambda$ is the Liouville form $\lambda = \sum p_i dq_i$.

There is a double cover $\phi : S^3 \to \Sigma$ such that, if our metric on $S^2$ is suitably normalized, we have $\phi^* \beta = \alpha$. The map $\phi$ extends to give a double cover
$$\Phi : (0, \infty) \times S^3 \to (0, \infty) \times \Sigma, \, (t, z) \mapsto (t, \phi(z))$$
which is symplectic relative to the forms $d(t \alpha)$ and $d(t \beta)$.

There are diffeomorphisms
$$e : (0, \infty) \times S^3 \to \CC^2 \setminus \{ 0 \}, \, (t, z) \mapsto \sqrt{t}z$$
and
$$s:  (0, \infty) \times \Sigma \to T^* S^2 \setminus \OO, \, (t, p) \mapsto tp$$
satisfying $\frac 12 e^* \sum (x_i dy_i - y_i dx_i) = t \alpha$ and $s^* \lambda = t \beta$. Here $\OO$ denotes the zero section. Then $s \circ \Phi \circ e^{-1}$ is a symplectic double cover from $\CC^2 \setminus \{ 0 \} \to T^* S^2 \setminus \OO$.

Starshaped domains in $\CC^2$ and $T^* S^2$ correspond under these diffeomorphisms to graphs over $S^3$ or $\Sigma$ in $ (0, \infty) \times S^3$ and $ (0, \infty) \times \Sigma$ respectively.

We notice that if a starshaped domain $U$ in $\CC^2$ is invariant under $z \mapsto -z$ then $s \circ \Phi \circ e^{-1} : U \setminus \{ 0 \} \to T^* S^2$ is a double cover and adding $\OO$ to the image we get a starshaped domain in $T^* S^2$. As all of our starshaped domains $X_v$ are toric, they have this invariance and we can define a family of starshaped domains $Y_v$ in $T^* S^2$ parameterized by $v \in \tilde{Q}$. Hence $Y_v \setminus \OO$ is double covered by $X_v \setminus \{ 0 \}$.

Our correspondence from domains in $\CC^2$ to domains in $T^* S^2$ is equivariant with respect to scaling, where for domains in $\CC^2$ scaling by $T$ denotes the map $z \mapsto \sqrt{T}z$, and for domains in $T^* S^2$ we mean the map $p \mapsto Tp$. Since scaling realizes (roughly) optimal embeddings between the $X_v$, and also gives embeddings between the $Y_v$, we see that
$$  d^{T^* S^2}_{SM}(Y_v, Y_w) \sim \le d_{SM}(X_v, X_w).$$
We would like to prove the reverse inequality. For this, suppose there exists a symplectic embedding $f: Y_v \hookrightarrow Y_w$. Then the nearby Arnold conjecture for Lagrangian spheres, see \cite{hi}, \cite{hi2}, \cite{blw}, implies we can find a Hamiltonian diffeomorphism $h$ of $T^* S^2$ with $h(f(\OO)) = \OO$.  In fact, as $Y_w$ is starshaped, composing with rescalings we may assume the isotopy of $\OO$ remains in $Y_w$ and hence that $h$ may be chosen with compact support in $Y_w$. Thus without loss of generality we may assume $f(\OO) = \OO$. Next, as Hamiltonians which vanish on $\OO$ restrict to give the diffeomorphisms of $\OO$ isotopic to the identity, we may further assume $f$ restricts to either the identity or the antipodal map on $\OO$. Finally, by an application of Moser's Lemma, as in the the proof of Weinstein's Lagrangian neighborhood theorem, we assume that $f$ restricts either to the identity or to the differential of the antipodal map on a neighborhood of $\OO$.

Using our double cover, the restriction of $f$ to the complement of $\OO$ lifts to a symplectic embedding $g : X_v \setminus \{0 \} \to X_w \setminus \{ 0 \}$, and near $0$ the embedding $g$ is either the identity map, or (if $f$ is antipodal on $\OO$) multiplication by $i$. Hence $g$ extends to a map $X_v \to X_w$ and we conclude
$$ d^{T^* S^2}_{SM}(Y_v, Y_w) \ge d_{SM}(X_v, X_w)$$ as required.
\proofend
\end{proof}

\subsection{Large-scale differences from scalings}

We conclude by providing the promised proof of the result mentioned in Remark~\ref{rem:notsame}, that there is a large-scale difference between scaling and symplectic embedding.  We refer to the pseudo-metric $d_I$ defined there.

\begin{proposition}
\label{prop:notsame}
Let $Z$ denote the space of concave toric domains in $\mathbb{R}^4$.  The identity map 
\[ (Z,d_{SM}) \to (Z,d_I)\]
is not a quasi-isometry. 
\end{proposition}

\begin{proof}
Fix a small $\eps > 0$, let $X_1 = B^4(1)$ and let $X_2$ be the concave toric domain associated to 
$B(1) \cup E(\eps,1/\eps)$ as in Section~\ref{sec:build}.    
Then 
\[ d_I(X_1,X_2) = ln(1+1/\eps).\]
On the other hand, $X_1$ includes into $X_2$.  As for $X_2$ into $X_1$, $X_2$ is a concave toric domain and $X_1$ is a convex one.  So, ECH capacities are sharp for this problem by \cite{CG}.  Since $E(\eps,1/\eps)$ embeds into $B^4(1)$ for $\eps$ sufficiently small \cite{mcs}, the ECH capacities of $X_2$ are bounded from above by the ECH capacities of an $E(1,2)$ by the results summarized in Section~\ref{sec:ECH}.
In particular, 
\[ d_{SM}(X_1,X_2) \le ln(2).\]
Since the choice of small $\eps > 0$ is arbitrary, we conclude that the identity map can not be a quasi-isometry.
\proofend
\end{proof}


\begin{thebibliography}{99}


\bibitem{blw} S. Borman, T-J. Li, W. Wu, 
{\em Spherical Lagrangians via ball packings and symplectic cutting}, Selecta Math. 20 (2014), 261--283.




\bibitem{ched} J. Chaidez and O. Edtmair, {\em 3D convex contact forms and the Ruelle invariant},
Invent. Math. 229 (2022),  243--301.

\bibitem{jems} D. Cristofaro-Gardiner, V. Humiliere and S. Seyfaddini, {\em PFH spectral invariants on the two-sphere and the large scale geometry of Hofer’s metric}, JEMS, June 2023.
 
 \bibitem{CG} D. Cristofaro-Gardiner, {\em Symplectic embeddings from concave toric domains into convex ones}, J. Differential Geom. 112 (2019): 199--232.


\bibitem{cghr} D. Cristofaro-Gardiner, M. Hutchings and V. Ramos, {\em The asymptotics of ECH capacities}, Invent. Math. 199 (2015), 187--214.



\bibitem{ccghr} K. Choi, D. Cristofaro-Gardiner, M. Hutchings and V. Ramos, {\em Symplectic embeddings into four-dimensional concave toric domains}, J. Topol. 7 (2014), 1054--1076.


\bibitem{dgvz} J. Dardennes, J. Gutt, V. Ramos and J. Zhang, {\em Coarse distance from dynamically convex to convex}, arXiv:2308.06604.

\bibitem{guth} L. Guth, {\em Symplectic embeddings of polydisks}, Invent. Math. 172 (2008), 477--489.

\bibitem{ghr} J. Gutt, M. Hutchings and V. Ramos, {\em Examples around the strong Viterbo conjecture},
J. Fixed Point Theory Appl. 24 (2022), 22 pp.




\bibitem{hi} R. Hind, {\em Lagrangian spheres in $S^2 \times S^2$}, Geom. Funct. Anal. 14 (2004), 303--318.

\bibitem{hi2} R. Hind, {\em Lagrangian unknottedness in Stein surfaces},
Asian J. Math. 16 (2012), 1--36.

\bibitem{mcs} D. McDuff and F. Schlenk, {\em The embedding capacity of 4-dimensional symplectic ellipsoids}, Ann. Math 175.3 (2012), 1191-1282.

\bibitem{h1} M. Hutchings, {\em Quantitative embedded contact homology}, J. Differential Geom. 88 (2011), 231--266.

\bibitem{h2} M. Hutchings, {\em Elementary ECH capacities}, Proc. Natl. Acad. Sci. (2022).

\bibitem{hutchings} M. Hutchings, {\em Lecture notes on Embedded Contact Homology}, Bolyai Soc. Math. Stud., 26 (2014), 389--484.

\bibitem{melistas} T. Melistas, {\em The contact Banach-Mazur distance and large scale geometry of overtwisted contact forms}, arXiv:2009.07369
















\bibitem{ps} L. Polterovich and E. Shelukhin, {\em Lagrangian configurations and Hamiltonian maps}, Comp. Math., 159 (2023), 2483-2520.

\bibitem{rosenzhang} D. Rosen and J. Zhang, {\em Relative growth rate and contact Banach-Mazur distance}, Geom. Dedicata 215 (2021), 1--30.



\bibitem{sz} V. Stojisavljevi\'{c} and J. Zhang, {\em Persistence modules, symplectic Banach-Mazur distance and Riemannian metrics},
Internat. J. Math. 32 (2021), no. 7, Paper No. 2150040, 76 pp.

\bibitem{usher} M. Usher, {\em Symplectic Banach-Mazur distances between subsets of $\CC^n$}, J. Topol. Anal. 14 (2022),  231--286.

\end{thebibliography}
\end{document}